\documentclass[10pt]{amsart}

\usepackage{amsmath,amssymb}
\usepackage{amscd}
\usepackage{amsfonts}
\numberwithin{equation}{section}
\theoremstyle{plain}
\newtheorem{thm}{Theorem}[subsection]
\newtheorem{prop}[thm]{Proposition}
\newtheorem{lem}[thm]{Lemma}

\newtheorem{cor}[thm]{Corollary}

\theoremstyle{definition}
\newtheorem{defn}[thm]{Definition}
\newtheorem{rmk}[thm]{Remark}
\newcommand{\mathbbm}[1]{\mathbf{#1}}

\newcommand{\F}{\mathbb{F}}
\newcommand{\Fhat}{\widehat{F}}
\newcommand{\Ga}{\mathbb{G}_a}

\newcommand{\kappawpbar}{\overline{\kappa(\wp)}}

\newcommand{\wpbar}{\overline{\wp}}

\newcommand{\Z}{\mathbb{Z}}

\newcommand{\cC}{\mathcal{C}}

\newcommand{\cK}{\mathcal{K}}

\newcommand{\cM}{\mathcal{M}}
\newcommand{\cO}{\mathcal{O}}
\newcommand{\cS}{\mathcal{S}}

\newcommand{\Awphatur}{\widehat{A}_\wp^{\mathrm{ur}}}

\newcommand{\Aut}{\mathrm{Aut}}

\newcommand{\CH}{\mathrm{CH}}

\newcommand{\End}{\mathrm{End}}

\newcommand{\GL}{\mathrm{GL}}

\newcommand{\Hom}{\mathrm{Hom}}
\newcommand{\id}{\mathrm{id}}

\newcommand{\Image}{\mathrm{Image}}

\newcommand{\Ker}{\mathrm{Ker}}

\newcommand{\op}{\mathrm{op}}

\newcommand{\Spec}{\mathrm{Spec}\,}

\newcommand{\Spf}{\mathrm{Spf}\,}

\newcommand{\Supp}{\mathrm{Supp}\,}

\newcommand{\frm}{\mathfrak{m}}

\newcommand{\inj}{\hookrightarrow}
\newcommand{\surj}{\twoheadrightarrow}
\newcommand{\resp}{resp.\ }
\newcommand{\xto}[1]{\xrightarrow{#1}}
\newcommand{\wt}[1]{\widetilde{#1}}
\newcommand{\wh}[1]{\widehat{#1}}

\newcommand{\advertisement}[1]{}

\newcommand{\bd}{{\mathbbm{d}}}

\renewcommand{\mod}{\,\mathrm{mod}\,}

\newcommand{\fram}{\mathfrak{m}}

\newcommand{\bF}{\mathbf{F}}

\newcommand{\Sym}{\mathrm{Sym}}

\title[Regularity of quotients]{Regularity of quotients of
Drinfeld modular schemes}
\author[Satoshi KONDO, Seidai YASUDA]
       {Satoshi KONDO$^{1,2}$, Seidai YASUDA$^3$       \\
       $^1${Middle East Technical University, Northern Cyprus Campus, Turkey}
\\
       $^2$Kavli Institute of the Physics and Mathematcis of the Universe, University of Tokyo (WPI), Japan
\\
       $^3$Department of Mathematics, Graduate School of Science, Osaka University, Japan
}

\email{satoshi.kondo@gmail.com}

\address{Corresponding author: 
Satoshi Kondo \\
Middle East Technical University \\
Northern Cyprus Campus, TZ-32, Kalkanli \\
Guzelyurt, Mersin 10, Turkey
\\
Email: satoshi.kondo@gmail.com\\
Kavli Institute for the Physics and Mathematics of the Universe\\
University of Tokyo\\
5-1-5 Kashiwanoha\\
Kashiwa 277-8583\\ Japan\\  
Tel: +81-4-7136-4940\\
Fax: +81-4-7136-4941\\
}
\address{Seidai Yasuda\\
Department of Mathematics\\
Graduate School of Science\\
Osaka University\\
1-1 Machikaneyama Toyonaka Osaka 
560-0043 JAPAN\\
Email:s-yasuda@math.sci.osaka-u.ac.jp}

\date{\today}

\begin{document}

\begin{abstract}
Let $A$ be the coordinate ring of 
a projective smooth curve over a finite field minus a closed point.
For a nontrivial ideal $I \subset A$, 
Drinfeld defined the notion of structure of level $I$ on
a Drinfeld module.

We extend this to that of level $N$, where
$N$ is a finitely generated torsion $A$-module.
The case where $N=(I^{-1}/A)^d$,
where $d$ is the rank of the Drinfeld module,
coincides with the structure of level $I$.
The moduli functor is representable by a regular affine 
scheme.

The automorphism group $\Aut_{A}(N)$ acts on the moduli space.
Our theorem gives a class of subgroups for which the quotient
of the moduli scheme is regular.   Examples include generalizations
of $\Gamma_0$ and of $\Gamma_1$.   
We also show that parabolic subgroups appearing in the 
definition of Hecke correspondences are such subgroups.
\end{abstract}

\maketitle
\section{Introduction}
\subsection{Main theorem and applications}
We first recall the usual setup for Drinfeld modules.
Let $C$ be a smooth projective geometrically irreducible curve over the finite field
$\F_q$ of $q$ elements.
Let $F$ denote the function field of $C$.
Fix a closed point $\infty$ of $C$.
Let $A=\Gamma(C \setminus\{\infty\}, \cO_C)$
be the coordinate ring of the affine $\F_q$-scheme
$C\setminus \{\infty\}$.

In this article, we define structure of level $N$ on
a Drinfeld module, generalizing the structure of level $I$ 
of Drinfeld (also known as the full level $I$ structure).
This may also be regarded as a generalization
of the $\Gamma_1$-structure.

Let us denote by $\cM_N^d=\cM_{N,A}^d$
the functor that associates an $A$-scheme 
$S$ to the set of isomorphism classes
of Drinfeld modules over $S$ 
with structure of level $N$.
(We will also mention $\cM_{N, U}^d$ 
for an open subscheme $U$ of $\Spec A$.)
The representability by an affine scheme
and its regularity of the moduli functor
when $|\Supp N| \ge 2$
can be proved in a manner similar to that of the
full level case.   See Proposition~\ref{prop:level N moduli}.

Note that the automorphism group
$\Aut_A(N)$ of $N$ as an $A$-module
acts on $N$, hence on the set of 
level structures,
and thereby on the moduli space.
We define admissible subgroups of 
$\Aut_A(N)$
and denote by $\cS_N$ the set of 
admissible subgroups.
Our main theorem in a rough form is as follows.
See Theorem~\ref{thm:main} 
for more general statement.
\begin{thm}
\label{thm:intro}
Let $N$ be a torsion $A$-module,
generated by at most $d$ elements.
Let $N=N_1 \oplus N_2$ be 
a direct sum decomposition such that $N_1 \neq 0$,
 $N_2 \neq 0$, and 
$\Supp N_1 \cap \Supp N_2=\emptyset$.
Let $H \in \cS_{N_1}$ be
an admissible subgroup,
regarded as a subgroup of $\Aut_A(N)$.
Then, the quotient
$\cM_N^d/H$
is regular. 
\end{thm} 
\noindent

Recall that moduli of elliptic curves
with $\Gamma_0$ or $\Gamma_1$
structures,
usually denoted $Y_0(N)$ or $Y_1(N)$ 
for some 
integer $N$,
are quotients of the 
full level moduli $Y(N)$ by certain 
groups defined in terms of congruences modulo $N$.
Our admissible subgroups include 
function field analogue
of such groups.


The quotients of interest to us are those
appearing in the definition of Hecke correspondences.
Let $Y \xleftarrow{f} X \xrightarrow{g} Y$ be a diagram of schemes where $Y$ is smooth over a base field and $f$ is a finite surjective morphism.   Recall that this gives rise to a finite correspondence (see \cite[p.3]{MVW} for finite correspondence).
A typical example of a
Hecke correspondence is the finite correspondence
corresponding to the 
diagram above where $X$ is
a quotient by some parabolic subgroup of a full level moduli and 
$Y$ is a full level moduli.
(See \cite[p.14]{Laumon} for Hecke correspondences in general over $F$; see Section~\ref{sec:application} for examples.)
A finite correspondence acts on higher Chow groups \cite[p.142, Thm 17.21]{MVW} and the action corresponding 
to a Hecke correspondence is called Hecke operator.

Our theorem may make the computation of the Hecke operator easier 
in the following sense.
Our regularity theorem implies that $X$ is regular (smooth over the base field)
and the maps $f, g$ are finite flat.   In this case, the action of the finite correspondence
is given as the composition of the pullback by $g$ and the pushforward by $f$.

In our future paper, we will construct certain elements in
the higher Chow groups of Drinfeld modular schemes over $A$,
extending, in a sense, our elements \cite{KY:zeta}
in rational algebraic 
$K$-theory of Drinfeld modular schemes 
away from level.
We show that they are (again) 
norm compatible (we 
say that they form an Euler system), that is,
we express the pushforward of the elements 
in terms of Hecke actions relating to local 
$L$-factors.    
The computation uses Theorem \ref{thm:intro} 
in a way described in the previous paragraph.

\subsection{Outline of proof}
An outline of proof of regularity of quotients 
of moduli of elliptic curves 
is given
by Katz and Mazur in \cite[Thm 7.5.1, p.201]{KM}.
We follow their outline except that we need 
one additional ingredient, namely, 
Dickson's theorem \cite{Dickson}
from modular invariant theory.

Let us briefly recall the outline.
To prove the regularity of the quotient,
one first proves the regularity 
at the points away from the level.
The regularity in this case is the consequence 
of the fact that the quotient map 
becomes \'etale.
Now, to prove the regularity at a point $y$
on the bad fiber,
we look at the point $x$ 
lying above $y$.    
To the point $x$, there corresponds 
a Drinfeld module (with level $N$ structure),
and in turn a formal $\cO$-module.
It happens that the regularity at $y$
depends only on the height of 
the formal $\cO$-module.
Since the singular locus is closed, and 
there are points of arbitrary height near
the (most) supersingular point
(we prove its existence in Section~\ref{sec:existence}),
we are reduced to the case where the
formal $\cO$-module has maximal height.

To show the regularity of the ring of 
invariants in the supersingular case,
the two tools we use are
Proposition 7.5.2 of \cite{KM}
(Proposition~\ref{prop:KM} below)
and 
Dickson's theorem (Theorem~\ref{thm:Hewett}).

Note that a typical type of groups encountered 
by Katz and Mazur are subgroups of 
$\GL_2(\Z/N\Z)$ for some integer $N$,
and the essential ones are those contained 
in some Borel subgroup.    In this case,
successive application of their Proposition 7.5.2 
proves the regularity of the quotient.
A typical type in our case is a parabolic
subgroup of $\GL_d(A/I)$ for some ideal $I$ 
with $d \ge 2$.    In this case,
using Proposition 7.5.2 successively,
we are reduced to the computation of 
invariants by the Levi subgroups, 
each of which is $\GL_{d'}(\kappa)$ for
some finite field $\kappa$ and some positive integer $d'$.
This last case is the subject of 
modular invariant theory,
where Dickson's theorem is the
most basic theorem.

\subsection{Organization}
The paper is organized as follows.
We start with Section~\ref{sec:formal module}
on formal $\cO$-modules,
much like in the original paper \cite{Drinfeld} by Drinfeld.
We introduce structure of level $N$ 
for formal $\cO$-modules and construct
the universal deformation explicitly.
This will be useful in the proof of theorem.
The following Section~\ref{sec:divisible module} is 
on divisible $\cO$-modules 
with structure of level $N$.    
In Section~\ref{sec:level N moduli}, we define 
Drinfeld modules with structure of level $N$
and show, using the results of 
Sections~\ref{sec:formal module} and~\ref{sec:divisible module}, 
that the moduli 
is regular.
In Section~\ref{sec:admissible}, 
we define a set 
of subgroups of $\Aut_A(N)$, 
which we call admissible.   
The definition is formulated so that 
the proof of our theorem 
can be given using 
the tools (Proposition~\ref{prop:KM} 
and Dickson's theorem).
We give some examples.
In Section~\ref{sec:main}, we give a statement of 
our main theorem (Theorem~\ref{thm:main})
and its applications.   Hecke operators are 
discussed in this section.
In Section~\ref{sec:existence}, we prove the existence of
a supersingular point on the moduli.   
This is logically independent of other sections.
Section~\ref{sec:proof} is devoted to the proof of our main result.

{\bf{Acknowledgment}} 
During this research, the second author was supported by the JSPS 
KAKENHI Grant Number JP15H03610.
The authors thank Takehiko Yasuda for very useful discussions.

\section{Formal $\cO$-modules with structure of level $P$}
\label{sec:formal module}
Let $\wp \subsetneq A$
be a nonzero prime ideal.
Let $A_\wp$
be the ring of integers of the 
local field $F_\wp$ at $\wp$.
We fix a uniformizer $\pi \in A_\wp$.
Let $A_\wp^{ur}$
denote the ring of integers of the 
maximal unramified extension
and 
$\Awphatur$
denote the completion.
Let $\cO=A_\wp$.
\subsection{Definitions}
Let $\cC$ be the category 
of complete local $\Awphatur$-algebras 
with residue field $\kappawpbar := \Awphatur/(\pi)$, where
the morphisms are local $\Awphatur$-algebra homomorphisms.
Let $R \in \cC$ and 
let $(F,f)$ be a formal $\cO$-module (see \cite[p.563]{Drinfeld})
over $R$.
Let $P$ be a finitely generated torsion $\cO$-module.
\begin{defn}
A structure of level $P$
on $(F,f)$ is a morphism of $\cO$-modules
\[
\psi: P \to \frm_R,
\]
where $\frm_R \subset R$ is the maximal ideal
endowed with the $\cO$-module structure given by $(F,f)$,
such that 
the power series $f_\pi$ is divisible by
\[
\prod_{\alpha \in 
\Ker\, \pi: P \to P}
(x-\psi(\alpha)).
\]
\end{defn}

Let $(G, g)$ be a formal $\cO$-module over $\kappawpbar$
with (the unique) structure of level $N$. 
\begin{defn}
\label{def:deformation}
A deformation of $(G,g)$ is 
a formal $\cO$-module $(F,f)$
over $R$,
for some $R \in \cC$, with structure of level $P$ 
such that $(F, f) \mod \frm_R$ is isomorphic to $(G,g)$.
\end{defn}

Let $D_P=D_{d,P}$ denote the functor that associates
$R \in \cC$ with the set of isomorphism classes of 
deformations over $R$ of $(G, g)$.

\subsection{Universal deformation space of formal
$\cO$-modules over $\overline{\kappa(\wp)}$}
\label{sec:universal deformation}
The following is due to Lubin-Tate and Drinfeld, and
the details are given in \cite{GH}. 

Let $d \ge 1$.
We define a formal $\cO$-module
$\Fhat_d$ 
over the ring 
$\cO[[t_1,\dots, t_{d-1}]]$
of formal power series as follows.
As a formal group,
$\Fhat_d=\widehat{\Ga}$.
The action of $a \in \kappa(\wp) \subset \cO$
on $\Fhat_d$ is 
given by the power series
$f_a(X)=aX$,
and the action of $\pi$
is given by the power series
\[
f_\pi(X)=
\pi X+t_1 X^{q_\wp}+\cdots+t_{d-1} X^{q_\wp^{d-1}}+X^{q_\wp^d}.
\]
Set $F_d=\Fhat_d \otimes_{\cO[[t_1, \dots, t_{d-1}]]}
\kappa(\wp)$.
Then $F_d$ is a formal 
$\cO$-module of height $d$ over $\kappa(\wp)$.
By \cite[Proposition 1.6, p.\ 566]{Drinfeld},
any formal $\cO$-module over 
$\overline{\kappa(\wp)}$
of height $d$ is isomorphic to
$F_d \wh{\otimes}_{\kappa(\wp)}
\overline{\kappa(\wp)}$.
\begin{prop}(Lubin-Tate, Drinfeld, Gross-Hopkins)
\label{prop:LTDGH}
The formal $\cO$-module 
$\Fhat_d \wh{\otimes}_{\cO[[t_1, \dots, t_{d-1}]]}
\Awphatur[[t_1, \dots, t_{d-1}]]$
is the universal deformation of 
$F_d \wh{\otimes}_{\kappa(\wp)} \overline{\kappa(\wp)}$.
\end{prop}
\begin{proof}
This is \cite[p.45, Prop.12.10]{GH}.
\end{proof}

\subsection{Universal deformation space of formal $\cO$-module
over $\overline{\kappa(\wp)}$ of level $P$}
We need the description of the universal deformation ring
$D_P$
of $F_d \wh{\otimes}_{\kappa(\wp)} \overline{\kappa(\wp)}$
of level $P$.
Set $D_0=\cO[[t_1, \dots, t_{d-1}]]$
by abuse of notation.
\begin{prop}
\label{prop:formal}
Let $\{e_i\}_{1 \le i \le r} \subset P$ be 
a minimal set of generators of $P$ as $\cO$-module.
If $r>d$, then the functor $D_P$ is representable by
an empty scheme. Suppose that $r \le d$.
Then the functor $D_P$ is representable by
a regular $D_0$-algebra. We let, by abuse of notation, 
$D_P$ denote the representing ring.
Then the images of the $e_i$'s in $D_P$ by 
the universal level structure and 
$t_i$'s for 
$r \le i \le d-1$ 
form
a regular system of parameters.
\end{prop}

This statement follows essentially from the proof of the full level 
case treated in \cite{Drinfeld}.   
As we want to use the description of the 
ring $D_P$, we give some details below.

\begin{proof}
It suffices
to treat the case 
$P=A/\wp^{n_1} \oplus \cdots \oplus A/\wp^{n_r}$
with $r \le d$, $n_1 \le \cdots \le n_r$
and $e_i=\bar{1} \in A/\wp^{n_i} \subset P$
for $1 \le i \le r$ with 
$1 \le n_1 \le \dots \le n_r$.
We construct the rings $D_P$ 
by induction on the exponent (i.e., the $n_r$) of $P$.

Suppose the exponent is 1.
Then $P=(A/\wp)^r$ for some $1 \le r \le d$.
In this case, 
the claim holds true with
$D_P=L_r$ of \cite[p.572, Lemma]{Drinfeld}.

Suppose the claim holds true for $k< n_r$.
Write
\[
P=A/\wp^{n_1} 
\oplus \cdots \oplus 
A/\wp^{n_s}
\oplus
A/\wp^k
\oplus \cdots \oplus
A/\wp^k
\]
with $1 \le n_1 \le \cdots \le
n_s <k$.
Set 
\[
P'=A/\wp^{n_1} 
\oplus \cdots \oplus
A/\wp^{n_s}
\oplus
A/\wp^{k-1}
\oplus \cdots \oplus
A/\wp^{k-1}.
\]

By inductive hypothesis,
$D_{P'}$ is representable, and 
a regular system of parameters is given
by the images of 
$e_1, \dots, e_s, e_{s+1}, \dots, e_r$
via the universal level structure
and 
$t_r, \dots, t_{d-1}$.

Now, set
\[
D_P=D_{P'}[[\theta_{s+1}, \dots, \theta_r]]/
(
f_\pi(\theta_{s+1})-\psi(e_{s+1}),
\dots,
f_\pi(\theta_{r})-\psi(e_{r})
)
\]
where $\psi$ is the universal level structure of level $P'$.
Then, the claim holds true.
\end{proof}

\section{Divisible $\cO$-modules with structure of level $P$}
\label{sec:divisible module}

\subsection{}
Let $R \in \cC$.
We refer to \cite[p.574 C)]{Drinfeld}
for the definition of a divisible $\cO$-module over $R$.
Let $(F, f)$ be a divisible $\cO$-module over $R$.
The number $j$ is defined by 
$F/F_{loc} \cong \Spf R \times (\cK/\cO)^j$
where $\cK$ is the field of fractions of $\cO$.
Let $h$ denote the 
height (assumed to be finite) 
of the 
reduction of $F_{loc}$.

\begin{defn}
A structure of level $P$ on a divisible
$\cO$-module
$(F,f)$
 is an $\cO$-module homomorphism
\[
\phi: P \to \Hom_{\Spf R}(\Spf R, F)
\]
such that (1) There is a submodule $P_1 \subset P$
the restriction of $\phi$ to which is a structure of level
$P_1$ on the formal $\cO$-module $F_{loc}$, and
(2) The induced map
$P/P_1 \to F/F_{loc}$
is an injection.
\end{defn}

\subsection{}
Let $(G,g)$ be a divisible $\cO$-module
over $\overline{\kappa(\wp)}$
with structure of level $P$ such that
$G_{loc}$
has height $h$ and 
$G/G_{loc} \cong (\cK/\cO)^j$.
The deformation of $(G,g)$ is defined.
\begin{prop}
\label{prop:divisible}
The functor that sends $R \in \cC$
to the set of deformations of level $P$
of the divisible module $(G,g)$ over $R$
is represented by the ring
\[
E_{(G, g)} \cong D_{P_1}[[d_1, \dots, d_j]]
\]
where $P_1$ is the submodule that appears 
in the level $P$ structure of $(G, g)$.
\end{prop}
\begin{proof}
The proof is almost identical to that of 
\cite[Prop.4.5, p.574]{Drinfeld}, hence omitted.
\end{proof}

\section{Regularity of moduli of level $N$}
\label{sec:level N moduli}
\subsection{}
Let $S$ be an $A$-scheme.
Let $E \to S$ be a Drinfeld module over $S$.
Let $N$ be a torsion $A$-module.  
\begin{defn}
A structure of level $N$ on $E$ is a homomorphism
of $A$-modules 
$\psi:N \to \Hom_{S\text{-schemes}}(S, E)$
such that,
for any element $a \in A$, 
the Cartier divisor
\[
\sum_{x \in \Ker\, a:N \to N} \psi(x)
\]
of $E$ is a closed subscheme of $E[a]=\Ker(a:E \to E)$.
\end{defn}

\subsubsection{}
Let $I \subset A$ be a nonzero ideal
and $d$ be the rank of $E$.
Then the case where $N=(I^{-1}/A)^d$
is the structure of level $I$ as defined
by Drinfeld \cite{Drinfeld}.

\subsubsection{}
Let $N$ be a nonzero finitely generated
torsion $A$-module.
For a nonzero prime ideal $\wp$ of $A$,
let $N_{\wp}$ denote the $\wp$-primary
part so that
$N=\oplus_{\wp} N_{\wp}$
is the primary decomposition.
Let $\psi: N \to \Hom(S, E)$
be a map and 
let $\psi_\wp:N_\wp \to
\Hom(S, E)$ be the restriction
for each nonzero prime ideal $\wp$.
Then, $\psi$ 
is a structure of level $N$
if and only if 
each $\psi_\wp$ 
is a structure of level $N_\wp$.

\subsection{}
Let $U \subset \Spec A$
be an open subscheme.
Let $\cM_{N, U}^d$ denote 
the functor
\[
(U\text{-scheme})
\to
(Set)
\]
that sends a $U$-scheme $S$ to 
the set of isomorphism classes 
of Drinfeld modules of rank $d$ over $S$
with structure of level $N$.

\begin{prop}
\label{prop:level N moduli}
Let $N$ be a nonzero finitely generated torsion $A$-module. 
\begin{enumerate}
\item
Suppose 
$|\Supp N| \ge 2$.
Let $U \subset \Spec A$ 
be an open subscheme.
Then the functor $\cM_{N,U}^d$ is representable
by a regular affine $U$-scheme.
\item
Let $Z \subset \Supp N$ be a nonempty subset.
Let $U \subset \Spec A \setminus Z$ 
be an open subscheme.
Then the functor $\cM_{N,U}^d$
is representable by a regular affine 
$U$-scheme.
\end{enumerate}
\end{prop}
\begin{proof}
The representability and regularity can 
be proved in 
a manner 
similar to the case of structure of level $I$
as in the proof of \cite[p.576, Prop. 5.3]{Drinfeld}.
\end{proof}

\section{Admissible subgroups}
\label{sec:admissible}
The goal of this section is to define the set $\cS_N$ of subgroups
of $\Aut_A(N)$, which we call admissible.   

\subsection{}
Let $N$ be a finitely generated torsion $A$-module, generated by at most 
$d$ elements.
Let 
$N=\bigoplus_{\wp \in \Spec A, \wp\neq (0)} N_\wp$
be the primary decomposition of $N$.
We will define a set $\cS_{N_\wp}$ 
of subgroups of $\Aut_A(N_\wp)$ (automorphisms as $A$-module)
for each $\wp$.
Set $\cS_N=\prod_{\wp \in \Spec A, \wp\neq (0)}\cS_{N_\wp}$.
This $\cS_N$ is regarded as a set of subgroups of $\Aut_A(N)$ 
in a natural manner.

From here on, we fix a nonzero prime ideal $\wp$ of $A$ and 
assume that $N$ is a $\wp$-primary torsion $A$-module
generated by at most $d$ elements.

\subsection{} \label{subsec:admissible}
\subsubsection{}
Let us take an isomorphism
$N \cong A_1 \oplus \dots \oplus A_r$
where 
$1 \le r \le d,
A_i=A/\wp^{n_i},
1 \le n_i,
1 \le i \le r$.

We use the following description of 
$\Hom_A(N,N)$ by matrices:
\[
\Hom_A(N,N)=
\{
(\varphi_{i,j})_{1 \le i,j \le r}
|
\varphi_{i,j} \in \Hom_A(A_i, A_j)
\}.
\]
We have canonically 
$\Hom_A(A_i, A_j)= \wp^{n_{i,j}}/\wp^{n_j}$
where
$n_{i,j}=\max\{0, n_j-n_i\}$.
Let $m_{i,j} \in \Z$ with 
$n_{i,j} \le m_{i,j} \le n_j$
for
$1 \le i,j \le r$.
We set
\[
H_{(m_{i,j})}
=
\{
(\varphi_{i,j})_{1 \le i,j \le r}
|
\varphi_{i,j} 
\in
(\delta_{i,j}+\wp^{m_{i,j}})
/
\wp^{n_j}
\}
\]
where $\delta_{i,j}$ 
is the Kronecker delta,
and regard it as a subset 
of $\Hom_A(A_i, A_j)$.

\subsubsection{}\label{par:admissible_subgp}
We consider the following condition for a
subgroup $H \subset \Aut_A(N)$:
\[
(*1): H = H_{(m_{i,j})} \cap
\Aut_A(N) \text{ for some $m_{i,j} \in \Z$ with 
$n_{i,j} \le m_{i,j} \le n_j$}.
\]
\begin{defn}
We say that a subgroup $H \subset \Aut_A(N)$
is an admissible subgroup
if 
Condition (*1) is satisfied
for some direct sum decomposition
$N=A_1 \oplus \dots \oplus A_r$
where 
$1 \le r \le d,
A_i=A/\wp^{n_i},
1 \le n_i,
1 \le i \le r$. 
We denote by $\cS_N$
the set of admissible subgroups of $\Aut_A(N)$.
\end{defn}

\subsubsection{}\label{sec:523}
We introduce some more notation to 
investigate properties satisfied by admissible subgroups.
For an $A$-module $B$ and an ideal $I \subset A$,
we set $B[I]=\bigcap_{i \in I} \Ker[i: B \to B]$.   

We have
\[
\begin{array}{rl}
N[\wp]
&=
\wp^{n_1-1}/\wp^{n_1}
\oplus
\cdots 
\oplus
\wp^{n_r-1}/\wp^{n_r}
\\
&\cong
\kappa(\wp)
\oplus \dots \oplus
\kappa(\wp),
\end{array}
\]
where $\kappa(\wp)=A/\wp$.

Let $H = H_{(m_{i,j})} \cap
\Aut_A(N)$ be an admissible subgroup for
some direct sum decomposition
$N=A_1 \oplus \dots \oplus A_r$
where 
$1 \le r \le d,
A_i=A/\wp^{n_i},
1 \le n_i,
1 \le i \le r$. 
Set 
\[
K=\Image
[H \subset \Aut_A(N) 
\to
\Aut_A(N[\wp])]
\]
where the arrow is the canonical map.
Let $S$ denote the set of pairs $(i,j)$ of integers
with $1 \le i,j \le r$ satisfying $m_{i,j} \neq  n_j - n_i$.
Since the composite
$$
\wp^{m_{i,j}}/\wp^{n_j} \inj \wp^{n_{i,j}}/\wp^{n_j}
\cong \Hom_A(A_i,A_j) \to \Hom_A(A_i[\wp],A_j[\wp])
$$
is the zero map if and only if
$(i,j) \in S$, the group
$K$ is identified with the set of
invertible $r \times r$ matrices $B$
with coefficients in $A/\wp$ such that 
for any integers $i,j$ with $1 \le i,j \le r$, 
the $(i,j)$-th entry
of $B-1_r$ is equal to zero when $(i,j) \in S$.
Using that $K$ is a group, one can check that $S$ satisfies
the following property: if $(i,j),(j,k) \in S$ then
$(i,k) \in S$.
For $i,j \in \{1,\ldots,r\}$, let us write
$i \sim j$ if either $i=j$ or ($(i,j), (j,i) \in S$).
It is easy to see that $\sim$ gives an equivalence relation
on the set $\{1,\ldots,r\}$.
For $i \in \{1,\ldots,r\}$, the equivalence class of $i$
will be denoted by $\overline{i}$.
As is easily seen, this equivalence relation has the following
property, which we will use later: if $i \in \{1,\ldots,r\}$ satisfies
$(i,i) \not\in S$, then $\overline{i}$ is the singleton $\{i\}$.
Let us consider the quotient set $\{1,\ldots,r\}/\sim$
under this equivalence relation.
For $i,j \in \{1,\ldots,r\}$, we write $\overline{i} \le \overline{j}$
if either $i\sim j$ or $(i,j) \in S$ is satisfied. 
The property of $S$ mentioned above implies
that this condition depends only on the classes 
$\overline{i}$, $\overline{j}$ of $i$, $j$, and
the relation $\le$ gives a partial order on the set $\{1,\ldots,r\}/\sim$.
Let us choose a total order on $\{1,\ldots,r\}/\sim$
extending this partial order and write
$\{1,\ldots,r\}/\sim \, = \{R_1, \ldots, R_u \}$,
$R_1 < \cdots < R_u$.
For $s=1,\ldots,u$, let $d_s$ denote the cardinality of
the subset $R_s \subset \{1, \dots, r \}$.
By permuting the elements $1,\ldots,r$ if necessary, we may and will
assume that $R_1, \ldots, R_u$ satisfy the following condition:
\begin{itemize}
\item[(**):]
For $s=1,\ldots,u$, the set $R_s$ is equal to the set of integers
$i$ satisfying $d_1+ \cdots + d_{s-1} < i \le  d_1 + \cdots + d_s$.
\end{itemize}
For $i=0,\ldots,u$ set
$F_s = \bigoplus_{i=1}^{d_1+\cdots + d_s} \wp^{n_i-1}/\wp^{n_i} 
\subset N[\wp]$.
This gives an increasing filtration of $N[\wp]$
as $\kappa(\wp)$-vector space:
\[
\{0\} = F_0 \subsetneq F_1
\subsetneq \cdots
\subsetneq F_u = N[\wp].
\]
Let 
\[
P_{F_\bullet}
=
\{
g \in 
\Aut_{\kappa(\wp)}
(N[\wp])
\ |\ 
g(F_i) =F_i
\text{ for all $1 \le i \le r$}.
\}
\]
Then $K$ has the following property:
\[
(*2): K \subset P_{F_\bullet}.
\]

For $ 1\le i \le u$, 
let 
$L_i=
\Aut_{\kappa(\wp)}
(F_i/F_{i-1})
$
and regard them as quotients of 
$P_{F_\bullet}$.

Let $R$ denote the set of integers
$s \in \{1,\dots,u\}$ such that
any $i,j \in R_s$ satisfies $(i,j) \in S$.
Then $K$ satisfies the following property:
\[
(*3):
\Image
\left[
K \subset P_{F_\bullet}
\to 
\prod_{1 \le i \le u} L_i
\right]
=
\prod_{i \in R} 
L_i,
\]
where 
$\prod_{i \in R} 
L_i$ is the trivial group
if $R= \emptyset$.

\subsection{Examples: 
$\Gamma_0$ and $\Gamma_1$}
Let $I \subsetneq A$ be a nonzero ideal.
Let $d \ge 1$.
Let $N=(A/I)^d$.
We consider the subgroup
$\Gamma_0$ (\resp $\Gamma_1$)
of $GL_d(A/I)$ consisting of elements $(a_{ij})_{1 \le i, j \le d}$
such that
\[
(a_{d, 1}, \dots, a_{d,d-1}) \equiv (0, \dots, 0) \mod I
\]
\[
\text{(\resp }(a_{d, 1}, \dots, a_{d,d}) \equiv (0, \dots, 0,1) \mod I).
\]
Then $\Gamma_0$ and $\Gamma_1$ belong to $\cS_N$.

\subsection{Examples: Parabolic subgroups}
\label{sec:parabolic}
Let $N=(A/I)^d$.
Let $\bd=(d_1, \dots, d_r)$ be a partition of 
$d=d_1+\dots+d_r$.
There is an associated parabolic subgroup 
$P_\bd \subset \GL_d(A/I)$.
Then $P_\bd$ is admissible.
\section{Main Theorem and its application}
\label{sec:main}
Let us state our main theorem and corollaries in this section.
The proof will be given in Section~\ref{sec:proof}.
\subsection{}
\begin{thm}
\label{thm:main}
Let $d \ge 1$.
Let $N$ be a torsion $A$-module generated by at most $d$ elements.
Suppose 
$N=N_1 \oplus N_2$
for some nonzero $N_1$ and $N_2$
such that 
$\Supp N_1 \cap \Supp N_2=\emptyset$.
Let $U \subset \Spec A$ be an open subscheme 
such that 
the pair $(N_2, U)$ satisfies Assumption (1) or (2) 
of Proposition~\ref{prop:level N moduli}.  
Let $H \subset \Aut_A(N_1)$ 
be a subgroup that belongs to 
$\cS_{N_1}$,
which is regarded as a 
subgroup of $\Aut_A(N)$.
Then, the quotient 
\[
\cM_{N, U}^d/H
\]
is regular.
\end{thm}


\begin{rmk}
The following case is not covered by Theorem~\ref{thm:main}.
Let $N$ be a torsion $A$ module
generated by at most $d$ elements
such that $\Supp N=\{\wp\}$ for 
some nonzero prime ideal $\wp$.
Let $U \subset \Spec A \setminus \Supp N$ 
be an open subscheme.
Then $\cM^d_{N, U}$ is representable
by Proposition \ref{prop:level N moduli}.
Take an admissible subgroup $H \in \cS_N$.
Then Theorem~\ref{thm:main} does not refer
to the quotient
$\cM^d_{N,U}/H$.
We think our proof will work if there exists 
an $A$-submodule $N_0 \subset N$
such that $N_0$ is $H$-invariant.
If there does not exist such an $A$-submodule,
we
do not know if the quotient is regular.
\end{rmk}

\subsection{}
Let $N_1', N_1'', N_2$ be torsion $A$-modules.
Suppose $N_1' \oplus N_2$ 
and $N_1'' \oplus N_2$
are generated by at most $d$ elements.
Assume that $\Supp N_1' \cap \Supp N_2=\emptyset$
and $\Supp N_1'' \cap \Supp N_2=\emptyset$.
Let $U \subset \Spec A$ be an open subscheme
such that the pair $(N_2, U)$
satisfies Assumption (1) or (2)
of Proposition~\ref{prop:level N moduli}.


\subsubsection{}
\label{sec:surjection}
Suppose we are given a  
surjective morphism of $A$-modules
$f: N'_1 \to N''_1$.
We write $f':N'_1 \oplus N_2 \to N''_1 \oplus N_2$
for the induced surjection.
Consider the functor that sends
a Drinfeld module over a scheme $S$
with structure of level $N'_1 \oplus N_2$ 
$(E \to S, \psi:N'_1 \oplus N_2 \to\Hom(S, E))$
to that with level $N''_1\oplus N_2$
$(E/\psi(N'_1 \oplus N_2) \to S,
\psi': N''_1 \oplus N_2 
\to \Hom(S, E/\Ker\, f'))$
where $\psi'(x)=\psi(y) \mod \Ker\, f'$
for any lift $y$ of $x$.
We denote the induced morphism
\[
m_f: \cM_{N'_1\oplus N_2, U}^d 
\to \cM_{N''_1 \oplus N_2, U}^d.
\]
As the morphism is finite (cf.\ \cite[p.8, (1.4.2)]{Laumon}) 
and both target and source are 
regular, we deduce that $m_f$ is flat
using \cite[V, p.95, 3.6]{AK}.

Let $H \subset \Aut_A(N'_1)$ be 
a subgroup.   Suppose $H$ is admissible,
i.e., $H \in \cS_{N'_1}$.   
We regard $H$ as a subgroup of 
$\Aut_A(N'_1 \oplus N_2)$
by letting it act trivially on $N_2$.
Let $H$ act on $N''_1$ trivially
and assume that $f'$ is a 
$H$-equivariant map.
Then the morphism $m_f$
factors as
$\cM_{N'_1\oplus N_2, U}^d\to 
\cM_{N'_1 \oplus N_2 , U}^d/H
\xto{h} \cM_{N''_1 \oplus N_2, U}^d$.
\begin{cor}
\label{cor:surjection}
The morphism $h$ is finite and flat.
\end{cor}
\begin{proof}
The morphism $m_f$ is finite.
Now, $\cM^d_{N'_1 \oplus N_2, U}/H 
\to \cM^d_{N''_1 \oplus N_2, U}$ 
is finite since 
$\cM^d_{N_2, U}$ is noetherian and 
$m_f$ is finite.
By Theorem~\ref{thm:main}, $\cM^d_{N'_1 \oplus N_2, U}/H$ is regular.
We use the fact (\cite[V, p.95, 3.6]{AK}) that 
a finite morphism between regular schemes of the same dimension
is flat to conclude.  
\end{proof}

\subsubsection{}
\label{sec:injection}
There is an analogous corollary for injections.
Let $N'_1, N''_1, N_2$, and $U$ as above.
Suppose we are given an injective morphism
$f:N''_1 \to N'_1$ of $A$-modules.
Let $f':N''_1 \oplus N_2
\to N'_1 \oplus N_2$
be the induced map.
Consider the functor that sends
a Drinfeld module over a scheme $S$
with structure of level $N'_1 \oplus N_2$ 
$(E \to S, \psi:N'_1 \oplus N_2 \to\Hom(S, E))$
to that with level $N''_1\oplus N_2$
$(E \to S, \psi': N''_1 \oplus N_2 \to \Hom(S, E))$
where $\psi'$ is the restriction of $\psi$.
We let \[
r_f: \cM_{N'_1\oplus N_2, U}^d 
\to \cM_{N''_1 \oplus N_2, U}^d.
\]
denote the induced morphism
of schemes.

Let $H \in \cS_{N'_1}$
be an admissible subgroup.
Let $H$ act on $N_2$ and $N''_1$
trivially and assume
$f': N''_1 \oplus N_2
\to
N'_1 \oplus N_2$
is $H$-equivariant.
Then the morphism $r_f$ factors as
$\cM_{N'_1\oplus N_2, U}^d\to 
\cM_{N'_1 \oplus N_2 , U}^d/H
\xto{h} \cM_{N''_1 \oplus N_2, U}^d$.
\begin{cor}
\label{cor:injection}
The morphism $h$ is finite and flat.
\end{cor}
\begin{proof}
The proof is analogous to the proof of Corollary~\ref{cor:surjection}, 
hence omitted.
\end{proof}

\subsection{Application: Hecke operators on higher Chow groups}
\label{sec:application}
With Corollaries~\ref{cor:surjection} and~\ref{cor:injection}
of previous section, we obtain the following 
description of Hecke operators on higher Chow groups 
of Drinfeld modular schemes over $A$.

\subsubsection{Hecke operators as finite correspondences}
Let $d \ge 1$.
Let $I \subsetneq A$ be a nonzero ideal.
Let $\wp \subset A$ be a prime ideal
which is prime to $I$.
Take $U$ to be an open subscheme of 
$\Spec A \setminus \Spec (A/I)$
if $|\Supp A/I|=1$ and 
$U$ to be any open subscheme of $\Spec A$
otherwise.

Set $N_{0,k}=(A/\wp)^k, N_2=(A/I)^d$
and $N_{1,k}=N_2 \oplus N_{0,k}$ 
for $0 \le k \le d$.
By Proposition~\ref{prop:level N moduli},
the functors
$\cM_{N_{1,k}, U}^d$ and 
$\cM_{N_2, U}^d$ 
are representable by regular schemes.

Let $G_k=\Aut_A(N_{0,k})$. 
Then, $G_k \in \cS_{N_{1,k}}$.
Regard $G_k$
as a subgroup of $\Aut_A(N_{1,k})$
with the identity on the direct factor $N_2$.


Let
$f_k: N_2 \to N_{1,k}$
be the canonical injection into the direct summand for each $k$.
As in Section~\ref{sec:injection}, we obtain a morphism
$r_{f_k} : \cM^d_{N_{1,k}, U} \to \cM^d_{N_2, U}$
which factors as
\[
\cM^d_{N_{1,k}, U}
\to
\cM^d_{N_{1,k}, U}/G_k
\xto{\overline{r}_{f_k}}
\cM^d_{N_2,U}.
\]
We denote by $\overline{r}_{f_k}$
the second morphism.

Let 
$g_k: N_{1,k} \to N_2$
denote the canonical surjection onto the direct summand for
for each $k$.
As in Section~\ref{sec:surjection},
we obtain a morphism
$m_{g_k} : \cM^d_{N_{1,k}, U} \to \cM^d_{N_2, U}$,
which factors as 
\[
\cM^d_{N_{1,k}, U}
\to
\cM^d_{N_{1,k}, U}/G_k
\xto{\overline{m}_{g_k}}
\cM^d_{N_2,U}.
\]
Let us denote by 
$\overline{m}_{g_k}$
the second morphism.

As we have seen in 
the proofs of Corollaries~\ref{cor:injection} and~\ref{cor:surjection},
without using the main results of this article, 
we know that the morphisms
$\overline{m}_{g_k}$ and $\overline{r}_{f_k}$
are finite.
Therefore 
the diagram
$\cM^d_{N_2, U} \xleftarrow{\overline{m}_{g_k}} \cM^d_{N_{1,k}, U}/G_k 
\xrightarrow{\overline{r}_{f_k}} \cM^d_{N_2, U}$
defines a finite correspondence in 
the sense used in \cite[p.142, Thm 17.21]{MVW}.
The action of this finite correspondence
on the higher Chow group
$\CH^*(\cM^d_{N_2,k}, *)$
is denoted by $T_{\wp, k}$
and we define this to be the $k$-th 
Hecke operator at $\wp$.

\subsubsection{}
As an application of our theorem, we can express 
Hecke operators as composition of pullback and pushforward
as follows.
As seen in Corollaries~\ref{cor:surjection} and~\ref{cor:injection},
the morphisms 
$\overline{r}_{f_k}$
and 
$\overline{m}_{g_k}$
are finite and moreover flat.
It follows that 
the graphs of 
$\overline{r}_{f_k}$
and 
$\overline{m}_{g_k}$
are finite correspondences,
and one can check that 
the composition equals $T_{\wp, k}$.
That is, we have
\[
T_{\wp, k}=(\overline{m}_{g_k})_*(\overline{r}_{f_k})^*: 
\CH^*(\cM_{N_2, U}^d,*) \to 
\CH^*(\cM_{N_2, U}^d,*),
\]
for each $0 \le k \le d$,
where upper star is the pullback and lower star is the pushforward.



\section{Existence of supersingular points}
\label{sec:existence}
In the proof of our main result, 
we use that there exists a supersingular point
at any $\wp$.
The aim of this section is to give a proof of that fact.

\subsection{}
Let $I \subset A$ be a nonzero ideal such that 
$|\Spec (A/I)| \ge 2$.
Let $\wp \subset A$ be a nonzero prime ideal 
and let $\kappa(\wp)=A/\wp$.
Let $\cM_{I,A}^d$ denote 
the moduli functor of full level $I$ Drinfeld modules of rank $d$.
The subscript $A$ indicates that the moduli is a functor from the 
category of $A$-scheme and it is representable by a scheme because of
the condition on $I$.
(We view it as a scheme.)
\begin{lem}
\label{lem:valued point}
We have $\cM_{I,A}^1(\overline{\kappa(\wp)}) \neq \emptyset$.
\end{lem}
\begin{proof}
Let $F_\infty$ be the completion at $\infty$
of $F$.
By \cite[Corollary, p.570]{Drinfeld}, 
there exists a Drinfeld module $E$ 
of rank 1 over 
$F_\infty^s$ where
$F_\infty^s$ is a separable closure of $F_\infty$.
As $F_\infty^s$ is separably closed, the $I$-torsion
points $E[I]$ of $E$ is isomorphic over 
$F_\infty^s$ to the constant 
$A$-module scheme $A/I$.   
A choice of an isomorphism
gives a level $I$ structure, thus we
see that $\cM_{I,A}^1(F_\infty^s)\neq 0$.
Since 
$\cM_{I, A}^1 \times_A F$ is an $F$-scheme
and $\overline{F} \subset \overline{F_\infty}$,
it follows that  
$\cM_{I,A}^1(\overline{F}) \neq \emptyset$.
Take a finite extension $L/F$ such that
$\cM_{I,A}^1(L) \neq \emptyset$.
Take a place $\wp_L$ over $\wp$.
Let $L_{\wp_L}$ denote the completion
of $L$ at $\wp_L$.
We have $\cM_{I,A}^1(L_{\wp_L}) \neq \emptyset$
using the canonical map $\Spec L_{\wp_L} \to \Spec L$.
Then, by \cite[Proposition 7.1, p.584]{Drinfeld},
there exists a finite extension $R$ of $L_{\wp_L}$
such that 
$\cM_{I, A}^1(\kappa) \neq \emptyset$ 
where $\kappa$ is the residue field of $R$.
This proves the claim.
\end{proof}
\subsection{Construction of a cover}
Let $C, \infty, A, F, F_\infty, C_\infty$ be as above.
Let $\wp \subset A$ be a nonzero prime ideal and 
$\kappa(\wp)=A/\wp$.

We construct a covering $C'$ of $C$ of degree $d$ as follows.
Let $f \in F$ be a nonzero element such
that $f$ has a zero of order 1 at each 
$\wp$ and $\infty$.
(The existence of such an $f$ can be proved 
by, for example, using the Riemann-Roch theorem.)

Set 
$F'=F[y]/(y^d-f)$
and let $C'$ be the smooth projective curve whose function field is $F'$.
Let $h: C' \to C$ denote the canonical map corresponding to
$F \subset F'$.
Then, by construction, $h$ is totally ramified at $\wp$ and $\infty$.
It follows that $h^{-1}(\infty)$ and
$h^{-1}(\wp)$ are singletons.
Let $\infty'$ and $\wp'$ denote the fibers of 
$\infty$ and $\wp$ respectively.

Set $A'=H^0(C' \setminus\{\infty'\}, \cO_{C'})$.
Recall that (e.g. as in \cite[p.33, Rmk 2.1]{DH})
we normalize the absolute values so that
$|a|_\infty=|A/(a)|$
and 
$|a'|_{\infty'}=|A'/(a')|$
for $a \in A$ and $a' \in A'$ respectively,
where $|\cdot|$ denotes the cardinality.
In particular,
we have
$|h_A(a)|_{\infty'}=|a|_\infty^d$
where $h_A: A \to A'$ is the map
induced by $h$.

\subsection{A rank 1 Drinfeld module for $(C', \infty')$ and its formal module}
\subsubsection{}
Let $I' \subset A'$ be an (auxiliary) nonzero ideal such that 
$|\Supp A'/I'| \ge 2$.
Take a nonzero prime ideal $\wp' \subset A'$.
By Lemma~\ref{lem:valued point} for $d=1$, 
we have 
$M^1_{I', A'}(\overline{\kappa(\wp')}) \neq \emptyset$.
Take $x' \in M^1_{I', A'}(\overline{\kappa(\wp')})$
and let $E'_{x'}$ denote the corresponding 
Drinfeld module over $\overline{\kappa(\wp')}$.
We write
$\varphi_{x'}: A' \to \End_{A'-group sch}(\mathbb{G}_{a,\overline{\kappa(\wp)}})$
for the corresponding ring homomorphism.
\subsubsection{}
The universal deformation of $E'_{x'}$
is computed in \cite[p.576, Section 5C]{Drinfeld}.
Let $\widetilde{E}'_{x'}$ denote the associated divisible $\widehat{A'_{\wp'}}^{ur}$-module,
and ${\widetilde{E'}_{x'}}^{loc}$ denote the connected component containing zero
(which is a formal $\widehat{A'_{\wp'}}^{ur}$-module).
This is isomorphic to the additive 
formal group with
$f_{\pi'}(x)=\pi'x+x^{|\pi'|_{\infty'}}$
where $\pi'$ is a uniformizer in $A'_{\wp'}$.
Hence the formal $\widehat{A'_{\wp'}}^{ur}$-module
associated with $E'_{x'}$ is isomorphic to the 
additive formal group with 
$f_{\pi'}(x)=x^{|\pi'|_{\infty'}}$.
\subsection{}
Using the ring homomorphism $\varphi_{x'}$, we construct a Drinfeld module
$(E, \varphi)$ for $(C, \infty)$ as follows.
Using the map $h_A: A \to A'$, we identify
$\overline{\kappa(\wp)}=\overline{\kappa(\wp')}$. 
We define a ring homomorphism $\varphi$ as the composite
\[
A \xto{h_A} A' \xto{\varphi_{x'}} 
\End_{A'-gp sch}(\mathbb{G}_{a, \overline{\kappa(\wp')}})
\to
\End_{A-gp sch}(\mathbb{G}_{a, \overline{\kappa(\wp)}}).
\]
It can be checked that this defines a Drinfeld module 
$(E, \varphi)$ for $(C, \infty)$
over $\overline{\kappa(\wp)}$.
The rank is $d$ since 
$\deg(\varphi(a))=|\varphi(a)|_{\infty'}=|a|_{\infty}^d$
for all nonzero $a$.

\begin{prop}
\label{prop:supersingular existence}
There exists a supersingular 
Drinfeld module 
(for $(C, \infty)$) 
of rank $d$
over $\overline{\kappa(\wp)}$.
\end{prop}
\begin{proof}
A candidate $(E, \varphi)$ was constructed above.
It remains to show that $E$ is supersingular.
Let us take for our uniformizer $\pi'$ 
the generator $y$ of $F'$ over $F$.
Then $\pi=(\pi')^d=y^d$ 
is a uniformizer in  
$A_\wp^{ur}$.
We have
\[
E[\wp]^{loc} \cong 
\hat{E}[\wp]
=\hat{E}[\pi]
\cong
\hat{E'}_{x'}[\pi'^d].
\]
Here, the superscript $\hat{\ }$ denotes the associated formal module.
The isogeny that has the last term as the kernel 
is of degree $|\pi|_\infty^d$.
Since 
the isogeny that has $E[\wp]$ as the kernel 
is of degree $|\pi|_\infty^d$, 
it follows that 
$E[\wp]^{loc}=E[\wp]$.
This implies that $E$ is supersingular.
\end{proof}

\section{Proof of Main Theorem}
\label{sec:proof}
\subsection{}
Let the notation be as in Theorem~\ref{thm:main}.

\subsection{One prime at a time}
Let the setup be as in Theorem~\ref{thm:main}.
Let $N_1=N_{1, \wp_1}
\oplus \cdots \oplus
N_{1, \wp_r}$
be the primary decomposition of $N_1$.
Let 
$H=H_{\wp_1} \times 
\cdots \times H_{\wp_r}$
be the decomposition 
given by the definition of 
admissible subgroup.   (In particular,
$H_{\wp_i} \in \cS_{N_{1,\wp_i}}$
for $1 \le i \le r$.)
Since
\[
\cM_{N,U}^d/H
=
\cM^d_{N_{1,\wp_1} \oplus N_2, U}/H_{\wp_1}
\times_{\cM_{N_2, U}^d}
\dots
\times_{\cM_{N_2, U}^d}
\cM^d_{N_{1,\wp_r} \oplus N_2, U}/H_{\wp_r},
\]
it suffices to treat the case where 
$N_1$ is $\wp$-torsion 
for some nonzero prime ideal $\wp$.

We assume from now that 
$N_1=N_{1, \wp}$ 
and
$H=H_{\wp_1} \in S_{N_1}$.

\subsection{Away from the prime $\wp$}
Let $x$ be a closed point of $\cM_{N,U}^d/H$.   
Take a closed point $y$ 
of $\cM_{N, U}^d$ 
that is sent to $x$ 
via the canonical quotient map
$\cM_{N,U}^d \to \cM_{N,U}^d/H$.
Let $U_{N_1}= 
U \cap (\Spec A \setminus \Supp N_1)=
U \setminus \{\wp\}$.
Note that the restriction to $U_{N_1}$ 
(the base change from $U$ to $U_{N_1}$)
$\cM^d_{N, U_{N_1}} \to 
\cM^d_{N, U_{N_1}}/H$
of the canonical quotient map is \'etale.
Therefore the regularity at $x$ follows from
the regularity at $y$, which in turn follows from
Proposition~\ref{prop:level N moduli}.

\subsection{At the prime $\wp$; dependence on the height}
We follow the outline given in \cite{KM}.

We use Section \ref{sec:surjection}
with $N'_1=N_1$, $N_1''=0$,
and $f:N'_1 \to N_1''$ the 
zero map.
Then we obtain a morphism
$m_f:
\cM_{N_1 \oplus N_2, U}^d 
\to 
\cM_{N_2, U}^d$.
This morphism factors as
\[
\cM_{N_1 \oplus N_2, U}^d
\to 
\cM_{N_1 \oplus N_2, U}^d/H
\to
\cM_{N_2, U}^d.
\]
(Recall $N=N_1 \oplus N_2$.)
Let 
$x \in \cM_{N,U}^d/H(\overline{\kappa(\wp)})$,
$y \in \cM_{N, U}^d(\overline{\kappa(\wp)})$
be a preimage of $x$,
and 
$z \in \cM_{N_2, U}^d(\overline{\kappa(\wp)})$
be the image of $x$.

Let $\cO_z$ be the local ring 
of $\cM^d_{N_2, U}
\times_U
\Spec\, A_\wp^\mathrm{ur}$
at $z$
and $U_z$ be the completion of $\cO_z$.
By \cite[p.576 C]{Drinfeld},
the ring $U_z$ is isomorphic
to the deformation ring of the formal
$\cO$-module with level $N_2$ structure
associated with the Drinfeld module 
corresponding to the point $z$
(see Proposition~\ref{prop:divisible}
for a description of the 
corresponding divisible $\cO$-module
and hence of the formal $\cO$-module).
We note that $U_z$ depends only 
on the height of the associated formal
$\cO$-module.

Let us consider the following commutative 
diagram, where each of 
the squares is cartesian:
\[
\begin{CD}
\cM^d_{N,U}   
@<<<
\cM^d_{N,U}\times_{\cM^d_{N_2,U}} \Spec \cO_z
@<<<
\cM^d_{N,U} \times_{\cM^d_{N_2, U}} \Spec U_z
\\
@VVV @VVV @VVV
\\
\cM^d_{N,U}/H
@<<<
\cM^d_{N,U}/H \times_{\cM^d_{N_2,U}} \Spec \cO_z
@<<<
\cM^d_{N,U}/H \times_{\cM^d_{N_2, U}} \Spec U_z
\\
@VVV @VVV @VVV
\\
\cM^d_{N_2,U}   
@<<<
\Spec \cO_z
@<<<
\Spec U_z.
\end{CD}
\]
We note that the bottom horizontal arrows are 
flat, hence all horizontal arrows are flat.

The regularity of $\cM^d_{N, U}/H$
at the image of the morphism 
from $\Spec(\overline{\kappa(\wp)})$
corresponding
to the $\overline{\kappa(\wp)}$-valued 
point $x$
is equivalent to the 
regularity of 
$\cM_{N,U}^d/H \times_{\cM^d_{N_2,U}} \Spec U_z$
at the unique point over $x$ because the 
morphism $\Spec U_z \to \cM^d_{N_2, U}$
is regular.

It follows from 
\cite[p.217, Prop A7.1.3(1)]{KM}
that
\[
(\cM^d_{N,U}/H 
\times_{\cM^d_{N_2, U}}
\Spec U_z)
\cong
(\cM^d_{N,U} 
\times_{\cM^d_{N_2, U}}
\Spec U_z)^H.
\]
Notice that this scheme on the right
depends only on the height of the associated 
formal module corresponding to $z$.
Thus the regularity at $x$ depends
only on the height.
\begin{rmk}
We remark that the isomorphism above 
is an analogue of \cite[p.194 Remark 7.1.4]{KM}.
The idea for the argument above is also taken from loc.\ cit.
The main difference is that we do not use relative moduli
problems.
\end{rmk}

\subsection{Reduction to supersingular case}
The height $h$ of the 
formal $\cO$-module 
corresponding to the point $y$ 
ranges from $1$ to $d$. 
(Note that this height is the same as
that corresponding to the point $z$.)
We call the point $y$ supersingular if the 
height equals $d$.
Note that in terms of 
$f_\pi$ (see Section~\ref{sec:universal deformation}),
$h$ is the number such that 
$t_1=\dots=t_{h-1}=0$
and $t_h \neq0$.

Now since the supersingular locus is nonempty
by Proposition~\ref{prop:supersingular existence}
and closed, 
and since there exists 
a point of 
arbitrary height near a supersingular point,
it suffices to treat the case $h=d$.

\subsection{Reduction to the standard case}
We now consider the case where the point $y$ is supersingular.
The completion of the local ring at $y$ 
is isomorphic to $D_{N_1}$ (see \cite[p.576]{Drinfeld}).   
The task is to show that the ring of 
$H$-invariants
$(D_{N_1})^H$ is regular.
The aim of this subsection is to reduce to the case
where $N_1=(A/\wp^n)^d$ for some $n$.

\subsubsection{}
Let 
$P=(\wp^{-n}/A)^d$
and 
$Q=\wp^{-n_1}/A
\oplus \cdots \oplus
\wp^{-n_r}/A$
for some 
$1 \le n_1, \dots, n_r \le n$, 
$1 \le r \le d$.
There is an inclusion
$Q \subset P$ 
induced by the inclusions
$\wp^{-n_i}/A \subset 
\wp^{-n}/A \subset P$
where the first is the canonical one
and the second is into the 
$i$-th summand for $1 \le i \le r$.
Let us take an  $A$-module $L$,  
whose $\wp$-primary component is zero, 
and an open subscheme $U \subset \Spec \, A$,
so that the pairs 
$(Q\oplus L, U)$ and 
$(P\oplus L, U)$
satisfy one of the assumptions 
in Proposition~\ref{prop:level N moduli}.

Let us write 
$Q'=Q \oplus L$ and 
$P'=P \oplus L$.
We then obtain a morphism
$Q' \subset P'$ 
induced by the inclusion 
$Q \subset P$
and the identity on $L$,
and hence a morphism
$r: \cM_{P', U}^d
\to \cM_{Q', U}^d$
using Section~\ref{sec:injection}.

Set 
$G_{P', Q'}=\{
g\in \Aut_A(P')
\ |\ 
gQ'=Q',
g|_{Q'}= \id_{Q'}
\}.
$
Then the morphism $r$ 
above factors as
\[
\cM_{P',U}^d
\to
\cM_{P',U}^d/G_{P', Q'}
\xto{h}
\cM_{Q',U}^d.
\]
\begin{lem} \label{lem:h_isom}
$h$ is an isomorphism.
\end{lem}
\begin{proof}
One can check directly that 
$h$ is an isomorphism 
away from $\wp$ (that is,
over $U\setminus \{\wp\}$).

Now, 
$\cM_{P', U}^d/G_{P', Q'}$ 
is normal, being a quotient of an affine regular scheme 
by a finite group action.
We also have that
$\cM_{Q',U}^d$ is regular hence normal.
As they agree away from $\wp$ and $h$ is finite,
being the normalizations, 
they agree over $U$.
\end{proof}

\subsubsection{}
Recall that $N_1$ is a $\wp$-torsion $A$-module
which is generated by at most $d$ elements.
Write $N_1=A_1 \oplus \cdots \oplus A_r$
where 
$A_i = A/\wp^{n_i}$ for $1 \le i \le r$
with some integer $n_i \ge 1$.
Let us choose an integer $n \ge \max_i n_i$.
Take an injection $A_i \inj A/\wp^n$ 
of $A$-modules for $i=1,\ldots,r$.
Set $\wt{N}_1=(A/\wp^n)^d$.
These injections give an injection 
$N \inj (A/\wp^n)^r
\inj (A/\wp^n)^d = \wt{N}_1$
where the second arrow is the injection
into the first $r$ factors.
For a subgroup 
$H \subset
\Aut_A(N_1)$, 
let $\wt{H}$ denote the subgroup
$\{g \in \Aut_A(\wt{N}_1)\ 
|\ g(N_1)=N_1, g|_{N_1} \in H\}$
of $\Aut_A(\wt{N}_1)$. 
As is easily seen, $H$ is an 
admissible subgroup of $\Aut_A(N_1)$ 
with respect to the direct sum decomposition as above 
if and only if
$\wt{H}$ 
is an admissible subgroup of $\Aut_A(\wt{N}_1)$.

From this and Lemma \ref{lem:h_isom}, we see 
that it suffices to treat the case
$N_1=(A/\wp^n)^d$.
\subsection{Some subgroups of $H$}
We now consider the case 
$N_1=(A/\wp^n)^d$
for some $n \ge 1$
in more detail.
We will need certain subgroups 
of an admissible subgroup $H$
for the proof of Theorem~\ref{thm:main}.
We label them below.

\subsubsection{}
To introduce subgroups of an 
admissible subgroup
$H\subset \Aut_A(N_1)$,
we introduce some more notation.
Let 
$e_i=\bar{1} \in 
A/\wp^n 
\subset (A/\wp^n)^d=N_1$
where the inclusion is 
into the $i$-th factor.
We regard 
$N_1=(A/\wp^n)^d$
as the set of 
row vectors,
on which
$\Aut_A(N_1)=
\GL_d(A/\wp^n)^\op$
acts as the multiplication from the right.

Let 
$R_{i,j} \subset A/\wp^n$
be subsets for
$1 \le i,j \le d$.
We use the notation
$\{(R_{i,j})\}$
to denote the subset
\[
\{
(r_{i,j})
\in M_d(A/\wp^n)
|
r_{i,j} \in
\delta_{i,j}
+R_{i,j}
\text{ for }
1\le i,j\le d
\}
\]
of the set 
$M_d(A/\wp^n)$ of $d$-by-$d$ matrices,
where $\delta_{i,j}$ 
is the Kronecker delta.

\subsubsection{}
Let
$H$
be an admissible subgroup
with respect to the standard direct sum
decomposition 
of $N_1$.
We assume that the subsets $R_1,\ldots,R_u$ of
$\{1,\ldots,d\}$ introduced in Section \ref{sec:523}
satisfy Condition (**). 
Then, by (*1),
there exists 
$(m_{i,j}) \in M_d(\Z)$
with 
$0 \le m_{i,j} \le n$
such that
\[
H=\{
(\wpbar^{m_{i,j}})
\}
\]
where $\wpbar=
\wp/\wp^n$
is the maximal ideal of $A/\wp^n$.
Let $K$ and $L_i$ ($i=1,\ldots,u$) be as in 
Section \ref{subsec:admissible} and set
$$
J = \Ker[H \surj K \to \prod_{1 \le i \le u} L_i].
$$
Let
$$
m'_{i,j} = \begin{cases}
1, & \text{ if }m_{i,j}=m_{j,i}=0, \\
m_{i,j} & \text{ otherwise}.
\end{cases}
$$
Then we have $J = \{(\wpbar^{m'_{i,j}})\}$.

Write $J_{i,j}=\wpbar^{m'_{i,j}}$.
Set 
$J_{i,j}^k=J_{i,j} \cap \wpbar^k$
for $0 \le k \le n$,
and 
$J^k=
\{
(J_{i,j}^k)
\}
$.
Set 
\[
J_{i,j}^{k,\ell}
=
\left\{
\begin{array}{ll}
J_{i,j}^k   
&
\text{if $i \le l$},
\\
J_{i,j}^{k+1}
&
\text{if $i > l$},
\end{array}
\right.
\]
for $0\le l \le d$ and 
$J^{k,\ell}=
\{(J_{i,j}^{k,\ell})\}$.
We have
\[
\begin{array}{ll}
J=J^0 = J^{0,d} \supset \cdots
\supset J^{0,0}
= J^1 = J^{1,d} \supset
\dots \supset
J^{1,0}
\\
=J^2=J^{2,d} \supset \dots
\supset J^{2,0 }
=J^3=\cdots.
\end{array}
\]
Note that 
\[
J^k=
\Ker[
J \subset
\GL_d(A/\wp^n)
\to
\GL_d(A/\wp^k)
]
\]
where the arrow is the canonical map,
and
\[
J^{k,\ell}
=
\{
h \in J^k
|
e_m h \equiv e_m \mod \wpbar^{k+1}
\text{ for }
\ell \le m \le d
\}.
\]

\subsubsection{}
It follows from this description that 
$J^{k+1}$
is a normal subgroup of 
$J^k$
for $k \ge 0$,
and the quotient
$J^k/J^{k+1}$
is abelian for $k \ge 1$.
Hence $J^{k,\ell-1}$
is a normal subgroup of 
$J^{k,\ell}$ for $k \ge 1, 1 \le l \le d$.
The situation is different for $k=0$. However
a similar statement also holds for $k=0$.
\begin{lem}
Let $1 \le \ell \le d$.
Then $J^{0,\ell-1}$
is a normal subgroup of $J^{0,\ell}$ and
the quotient $J^{0,\ell}/J^{0,\ell-1}$
is abelian.
\end{lem}

\begin{proof}
It suffices to prove that
$J^{0,\ell-1}/J^1$
is a normal subgroup of $J/J^1$ and
the quotient $(J^{0,\ell}/J^1)/(J^{0,\ell-1}/J^1)$
is abelian.

By definition we have $J^1 = \Ker[H \surj K]$.
Hence the surjection $H \surj K$ induces an isomorphism
$J/J^1 \cong \Ker[K \to \prod_{1 \le i \le u} L_i]$.
Via this isomorphism we regard $J/J^1$ as a subgroup of 
$K \subset \Aut_A(N_1[\wp]) = \GL_d(\kappa(\wp))^\op$.
In particular each $g \in J/J^1$ is an element of
$\GL_d(\kappa(\wp))$ and 
acts on $N_1[\wp] = (\wp^{n-1}/\wp^n)^d$, whose elements we regard
as row vectors, as the multiplication by $g$ from the right.
Note that the submodule 
$N_1[\wp]_{\ge \ell} = \bigoplus_{\ell \le i \le d}\wp^{n-1}/\wp^n$ 
of $N_1[\wp]$ is stable under the action of $J/J^1$.
It follows from the definition of $J^{0,\ell-1}$ that
$J^{0,\ell-1}/J^1$ is equal to the kernel of
$J/J^1 \to \Aut_A(N_1[\wp]_{\ge \ell})$.
This in particular shows that $J^{0,\ell-1}/J^1$ is a normal subgroup
of $J/J^1$.
Let us fix an element $e \in N_1[\wp]_{\ge \ell} \setminus
N_1[\wp]_{\ge \ell+1}$. Let us consider the map
$f: J^{0,\ell}/J^1 \to N_1[\wp]_{\ge \ell}$ that sends
$g \in J^{0,\ell}/J^1$ to $eg - e \in N_1[\wp]_{\ge \ell}$.
It is then easy to check that $f$ is a homomorphism of groups
and the kernel of $f$ is equal to $J^{0,\ell-1}/J^1$.
This shows that
the quotient $(J^{0,\ell}/J^1)/(J^{0,\ell-1}/J^1)$
is abelian and the proof is complete.
\end{proof}

\subsubsection{}
Set 
$Q^{k,\ell}
=J^{k,\ell}/
J^{k,\ell-1}$
for 
$k \ge 0, 1 \le l \le d$.
We take the set of representatives 
of $Q^{k,\ell}$
as follows.
Let us choose a uniformizer $\pi \in A_\wp$ and set
\[
Q_{i,j}^{k,\ell}
=
\left\{
\begin{array}{ll}
\{0\}
&
\text{if }
J_{i,j}^{k,\ell}
=J_{i,j}^{k,\ell-1},
\\
\{a \pi^k
|
a \in \kappa(\wp)
\}
&
\text{if }
J_{i,j}^{k,\ell}
\neq J_{i,j}^{k,\ell-1}.
\end{array}
\right.
\]
Then the set
$\{
(Q_{i,j}^{k,\ell})
\}$
is a subset of $J^{k,\ell}$
and is a complete set of representatives for $Q^{k,\ell}$.
Moreover, 
the image of $\{(Q_{i,j}^{k,\ell})\}$
under the homomorphism $J^{k,\ell} \to J^{k,\ell}/J^{k+1}$
is a subgroup of $J^{k,\ell+1}/J^{k+1}$.
%
%
%
%
%

\subsection{Some additive polynomials}
\subsubsection{}
We regard the $\wp$-torsion $A$-module $N_1$
as an $\cO$-module.   $A$-submodules
are regarded as $\cO$-submodules and
vice versa.

Recall (see Section \ref{sec:formal module})
that $D_{N_1}$ is the universal 
deformation ring of the formal $\cO$-module
$F_d \wh{\otimes}_{\kappa(\wp)} \overline{\kappa(\wp)}$ 
(see Definition \ref{def:deformation} and Proposition \ref{prop:LTDGH})
with maximal height equipped with a 
structure of level $N_1$.
The formal $\cO$-module is 
isomorphic as a formal group
to $\widehat{\mathbb{G}}_a$,
and we use $f$ as in 
Section \ref{sec:universal deformation}
to denote the $\cO$-structure.
The explicit description of $D_{N_1}$
is found in (the proof of)
Proposition \ref{prop:formal}.
We regard elements of $N_1$ as elements
of $\mathfrak{m}_{D_{N_1}} \subset D_{N_1}$
via the universal level structure
\[
\psi: N_1 \to \mathfrak{m}_{D_{N_1}}
\]
where $\mathfrak{m}_{D_{N_1}}$ 
is the maximal ideal of $D_{N_1}$.
(That is, we write $n_1$ 
to mean $\psi(n_1)$.)

We regard elements of $N_1$ as row vectors with coordinates in $A/\wp^n$.
Then the group $\GL_d(A/\wp^n)$ acts on $N_1$ from the right, where
the action of $g \in \GL_d(A/\wp^n)$ is given by the multiplication 
$-\cdot g : N_1 \to N_1$ by $g$ from the right.
Let $\cC$ be as in Section \ref{sec:formal module}.
Let $g \in \GL_d(A/\wp^n)$.
For a deformation $(F,f)$ over some $R \in \cC$ with
structure $\psi_F$ of level $N_1$, let 
$g \psi_F$ denote the composite $\psi_F \circ (-\cdot g)$.
Then $g \psi_F$ is another structure of level $N_1$.
By sending $(F,f,\psi_F)$ to $(F,f,g \psi_F)$
for each $(F,f,\psi_F)$, we obtain an automorphism
of the universal deformation ring $D_{N_1}$.
We denote this automorphism also by $-\cdot g$.
It is then straightforward to check that
the equality $\psi(x)\cdot g = \psi(x g)$
holds for any $x \in N_1$.

\subsubsection{}
Let us introduce some additive polynomials with 
coefficients in $D_{N_1}$ 
and list some of the properties.
Let 
$M \subset N_1=(A/\wp^n)^d$
be an $A$-submodule.
We set
\[
f^M(x)=
\prod_{\alpha \in M}
(x-\alpha)
\in D_{N_1}[x].
\]

As $M$ is an abelian group,
$f^M(x)$
is an additive polynomial,
that is $f^M(x+y)=f^M(x)+f^M(y)$
holds
where $x$ and $y$ are 
indeterminates.

Since $\kappa(\wp) \subset A_\wp=\cO$,
any $M$ as above is a $\kappa(\wp)$-vector
space.   Hence we have
$f^M(sx)=sf^M(x)$
for any $s \in \kappa(\wp)$.
%
%
It follows from the construction of $\wh{F}_d$ given in
Section \ref{sec:universal deformation} that
the action of $s \in \kappa(\wp) \subset \cO$
on $\fram_{D_{N_1}}$ as a formal $\cO$-module
is equal to the multiplication by $s$
in the $\overline{\kappa(\wp)}$-algebra $D_{N_1}$.
Hence we have
$f^M(s n)=sf^M(n)$
for any $s \in \kappa(\wp)$ and for any $n \in N_1$.

Let $y \in D_{N_1}$
and $g \in \GL_d(A/\wp^n)$.
Then we have
\[
(f^M(y)) \cdot g
=
f^{Mg}(yg)
=\prod_{\alpha \in Mg}
(yg - \alpha).
\]
%
%
%
%

\subsubsection{}
Now we look at the $\cO$-module structure
$f$ of the universal deformation 
$(D_{N_1}, f)$
(see Section~\ref{sec:universal deformation}).

For $z \in \kappa(\wp)$,
we have $f_z(x)=zx$.

Let $n \ge m \ge 0$.   
Then the power series (actually a polynomial in the case at hand)
\[
f_{\pi^m}(x) \in D_{N_1}[[x]]
\]
giving the multiplication-by-$\pi^m$ has 
as the set of roots the set of 
$\pi^m$-torsion points.
Thus we have
\[
f_{\pi^m}(x)=
f^M(x)
\]
where $M=N_1[\pi^m] \subset (A/\wp^n)^d$
is the set of $\pi^m$-torsion points.


As the universal level structure $\psi$ is
an $\cO$-module homomorphism, we have
$f_a(n_1)=an_1$
for $a\in \cO$ and $n_1 \in N_1$
(by abuse of notation, we write $n_1$
for $\psi(n_1)$).
Using the additivity, 
we have
$f_{c\pi^k}(n_1+n_1')=
f_{c\pi^k}(n_1)+f_{c\pi^k}(n_1')
=c\pi^k n_1+c\pi^k n_1'$
for $0 \le k \le n$,
$n_1, n_1' \in N_1$, and $c \in \kappa(\wp)$.

\subsection{A proposition of Katz and Mazur}
We recall the following proposition 
\cite[Prop.\ 7.5.2, p.205]{KM}.
The notations are those in \cite{KM}.
\begin{prop}[Prop. 7.5.2, p.205, \cite{KM}]
\label{prop:KM}
Let $A$ be a complete noetherian local ring
which is regular of dimension $n$
and whose residue field is perfect.
Let $G$ be a finite subgroup of 
$\Aut(A)$,
such that 
every $g \in G$
acts trivially on the 
residue field
of $A$.
Let $(x_1, \dots, x_{n-1}, y)$ be a regular system of 
parameters in $A$.
Assume that for each $g \in G$ we have
\[
\begin{array}{ll}
1.\ g(x_i)=x_i \text{ for } i=1, \dots, n-1,
\\
2.\ g(y)\equiv u y \mod (x_1, \dots, x_{n-1}) \text{ for some unit }
u\in A^\times.
\end{array}
\]
Then 
\begin{enumerate}
\item 
$A$ is free over $A^G$
with basis $1,y,y^2, \dots, (y)^{\#G-1}$.
\item
$A^G$ 
is a regular local ring of dimension $n$.
\item
A regular system of parameters for $A^G$ 
is 
$(x_1, \dots, x_{n-1}, N(y))$
where $N(y)$ is the norm
$\prod_{g \in G}g(y)$.
\end{enumerate}
\qed
\end{prop}

\begin{lem}\label{lem:rem_after_KM}
In the setting of Proposition \ref{prop:KM},
Assumption 2 follows from Assumption 1.
\end{lem}
\begin{proof}
Assumption 1 implies that the
action of $G$ on $A$ induces the action of $G$ on the quotient
ring $\overline{A} = A/(x_1,\ldots,x_{n-1})$. 
Since $x_1, \ldots, x_{n-1}, y$ form a regular system of parameters, 
$\overline{A}$ is a discrete valuation ring and
the image $\overline{y}$ of $y$ in $\overline{A}$ is
a uniformizer of $\overline{A}$.
Hence $g(\overline{y})$ is also a uniformizer 
and we have $g(\overline{y}) = u \overline{y}$ for some
unit $u$ of $\overline{A}$. Since any lift $\wt{u} \in A$
of $u$ is a unit in $A$, it follows that Assumption 2
is satisfied.
\end{proof}

\subsection{}

For $k \ge 0$ and $1\le i \le d$, we let
$J^k_i$ denote the $A$-submodule
$$
J^k_i = J^k_{i,1} \oplus \cdots \oplus J^k_{i,d}
$$
of $N_1$.
For $0\le \ell \le d$, we set
$$
J^{k,\ell}_i = J^{k,\ell}_{i,1} 
\oplus \cdots \oplus J^{k,\ell}_{i,d}.
$$

\begin{prop}
\label{prop:Hkl}
Let $k \ge 0$ and $0\le \ell \le d$.\\
1. The ring of invariants
$(D_{N_1})^{J^{k,\ell}}$
is regular.\\
2. $f^{J_i^{k,\ell}}(e_i)$ for $1 \le i \le d$
form a regular system of parameters in $(D_{N_1})^{J^{k,\ell}}$.
\end{prop}
\begin{proof}
We prove this inductively.
The groups $J^{k,\ell}$ are
ordered by inclusion (both $k$ and $\ell$ run).
For $k$ large, $J^{k,\ell}=\{1\}$, so the claim
holds true by Proposition~\ref{prop:formal}.

Let $1 \le \ell \le d$.  (The case $\ell=0$
appears as the case $\ell=d$.) 
Suppose the claim holds true for 
$J^{k,\ell-1}$.  We prove the claim for $J^{k,\ell}$.

We use Proposition~\ref{prop:KM} in the following manner.
Recall that $J^{k,\ell-1}$
is a normal subgroup of $J^{k,\ell}$.
We set $G=J^{k,\ell}/J^{k, \ell-1}$
and $A=(D_{N_1})^{J^{k, \ell-1}}$
so that 
$A^G=(D_{N_1})^{J^{k,\ell}}$.
We let
$y=f^{J_\ell^{k,\ell-1}}$,
$n$ in Proposition~\ref{prop:KM} to be $d$, 
and
\[
(x_1, \dots, x_{d-1})
=
(
f^{J_1^{k, \ell-1}}(e_1), \stackrel{\ell}{\dots}, 
f^{J_d^{k, \ell-1}}(e_d)
)
\]
(equality not as ideals but as tuples)
where the $\ell$ means omission of the $\ell$-th component.

\begin{lem} \label{lem:Lemma1}
With notation as above,
Assumptions 1 and 2 of Proposition \ref{prop:KM} hold.
\end{lem}
\begin{proof}
By Lemma \ref{lem:rem_after_KM},
it suffices to prove that Assumption 1 holds. 
%

By the inductive hypothesis, the elements
$f^{J_i^{k,\ell-1}}(e_i) \in (D_{N_1})^{J^{k,\ell-1}}$
for $1 \le i \le d$ form a regular system of parameters.
Let $i$ be an integer with $1 \le i \le d$, $i \neq \ell$.
We show that $f^{J_i^{k,\ell-1}}(e_i)$ is 
fixed under the action of $J^{k,\ell}$. 

Take $g \in Q^{k,\ell}=J^{k,\ell}/J^{k,\ell-1}$.
Take the representative $\wt{g}$ of $g$
in $\{(Q_{i,j}^{k,\ell})\}$.
Since $i \neq \ell$, we have $e_i \wt{g} = e_i$
and $J_i^{k,\ell-1} = J_i^{k,\ell}$.
Hence we have $f^{J_i^{k,\ell-1}}(e_i)=f^{J_i^{k,\ell}}(e_i)$
and $f^{J_i^{k,\ell-1}}(e_i)g = f^{J_i^{k,\ell} \wt{g}}(e_i)$.
Hence it suffices to show $J_i^{k,\ell} \wt{g} = J_i^{k,\ell}$.

Let $x \in J_i^{k,\ell}$. Let $A_i(x)$ denote the $d$-by-$d$ matrix
whose $i$-th row is $x$ and the rest is zero.
Then we have $I_d + A_i(x) \in J_i^{k,\ell}$, where $I_d$ denotes
the $d$-by-$d$ identity matrix. Observe that the $i$-th row of
the product $(I_d + A_i(x))\wt{g} \in J_i^{k,\ell}$ is equal to
$e_i + x \wt{g}$. This shows $x \wt{g} \in J_i^{k,\ell}$.
Since $x \in J_i^{k,\ell}$ is arbitrary, we have 
$J_i^{k,\ell} \wt{g} \subset J_i^{k,\ell}$.
Since $J_i^{k,\ell}$ is finite, we obtain the equality
$J_i^{k,\ell} \wt{g} = J_i^{k,\ell}$, as desired.
\end{proof}

With Lemma \ref{lem:Lemma1} above, we apply Proposition \ref{prop:KM}
to our situation.   This completes the proof of Proposition.
\end{proof}

\subsection{}
Using Proposition~\ref{prop:Hkl} with $k=1$ and $\ell=0$,
we see that the ring of invariants 
$(D_{N_1})^{J}$ is regular and that
$f^{J_i}(e_i)$ for $1 \le i \le d$ form a regular system of parameters.
As $H/J \cong \prod_{j \in R}L_j$, it remains to take the 
invariants under the action of $\prod_{j \in R}L_j$.
%
%

The group $H/J = \prod_{j \in R}L_j$ acts on 
$(D_{N_1})^{J}$.   To use Dickson's theorem below, 
we check the following lemma.
\begin{lem}
The action of $\prod_{j \in R}L_j$ on the subset 
$f^{J_1}(e_1), \dots, f^{J_d}(e_d)$
is $\kappa(\wp)$-linear.
\end{lem}
\begin{proof}
Take $g=(g_{c,d}) \in \prod_{j \in R}L_j$.
We have
\[
\begin{array}{ll}
f^{J_b}(e_b) \cdot g
=
f^{J_b g}(e_b g)
=
f^{J_b}(e_b g)
=
f^{J_b}(g_{b,1}e_1+\dots+
g_{b,d} e_d)
\\
=
f^{J_b}(g_{b,1}e_1)+\dots+
f^{J_b}(g_{b,d}e_d)
=
g_{b,1}f^{J_b}(e_1)
+\dots+
g_{b,d}f^{J_b}(e_d).
\end{array}
\]
This proves the claim.
\end{proof}

\subsection{}
We recall Dickson's theorem.
Let $V$ be a $d$ dimensional $\kappa(\wp)$-vector space.
Taking a basis 
$x_1, \dots, x_d$,
we identify the symmetric algebra
$\Sym_{\kappa(\wp)} V$
with the polynomial ring
$\kappa(\wp)[x_1, \dots, x_d]$.
Let $G=\GL(V)$.   
Let $\bF$ be 
a field containing $\kappa(\wp)$.
Then $G$ acts on
$\bF[x_1, \dots, x_d]=
\Sym_{\kappa(\wp)}(V) \otimes_{\kappa(\wp)} \bF$.
We will use the following theorem.
\begin{thm}[{\cite{Dickson}} (see {\cite[p.239, Thm 8.1.5]{Smith}})]
\label{thm:Hewett}
The ring of invariants 
$\bF[x_1, \dots, x_d]^G$
is a polynomial ring $\bF[f_1, \dots, f_d]$
for some explicitly given homogeneous polynomials 
$f_1, \dots, f_d$. \qed
\end{thm}

\begin{cor}
\label{cor:Dickson}
$\bF[[x_1, \dots, x_d]]^G
=\bF[[f_1, \dots, f_d]]$
\end{cor}
\begin{proof}
The inclusion of the right hand side into
the left hand side is obvious.   

Let $y \in \bF[[x_1, \dots, x_d]]^G$
and write 
$y=\sum_{i \ge 0} y_i$
where each $y_i$ is a homogeneous polynomial
of degree $i$ in the $x$'s.
By definition, $\sigma \in G$ sends a homogeneous
polynomial to a homogeneous polynomial of the same degree.
Therefore, each $y_i$ belongs to 
$\bF[x_1, \dots, x_d]^G$, hence to $\bF[f_1, \dots, f_d]$.

Let $g=c \prod_{i=1}^d f_i^{a_i} \in \bF[f_1, \dots, f_d]$
be a monomial where $c \in \bF$
and $a_i \ge 0$ for each $1 \le i \le d$.
Then $g$ is homogeneous of degree
$\sum_{i=1}^d a_i \deg f_i$ in the $x$'s.
Hence $g$ appears as a summand of $y_m$
only if $m=\sum_{i=1}^d a_i \deg f_i$.
In particular, $g$ appears only finitely many times.
Therefore $\sum_{i \ge 0} y_i$ defines an element of 
$\bF[[f_1, \dots, f_d]]$.
This gives the other inclusion.
\end{proof}

\subsection{}
\begin{proof}[Proof of Theorem \ref{thm:main}]
We identify
$(D_{N_1})^{J}$
with
$\overline{\kappa(\wp)}
[[e_1'', \dots, e_d'']]$
where we write
$e_i''=f^{J}(e_i)$ 
for $1 \le i \le d$.
Corollary~\ref{cor:Dickson}
implies that the 
ring of invariants by $\prod_{j \in R}L_j$ 
is again a ring of formal power series, hence
regular.
\end{proof}

\end{document}